\documentclass[11pt]{amsart}

\usepackage{amsmath,amssymb,amsthm}
\usepackage{graphicx}
\usepackage[numbers]{natbib}
\usepackage{hyperref}
\hypersetup{colorlinks=true, linkcolor=blue, citecolor=blue, urlcolor=blue}

\newtheorem{theorem}{Theorem}
\newtheorem{lemma}[theorem]{Lemma}
\newtheorem{proposition}[theorem]{Proposition}
\newtheorem{corollary}[theorem]{Corollary}
\theoremstyle{definition}
\newtheorem{definition}[theorem]{Definition}

\theoremstyle{remark}

\begin{document}

\title{Truncated Plethystic Exponentials Preserve Power Sum Constraints}
\author{Yogesh Phalak}
\address{Department of Mechanical Engineering, Virginia Tech, Blacksburg, VA 24061}
\email{yphalak@vt.edu}
\date{\today}

\begin{abstract}
Given an arbitrary sequence $(\alpha_1, \ldots, \alpha_n) \in \mathbb{C}^n$, we show that the degree-$n$ truncation of the formal exponential $\exp\bigl(-\sum_{k=1}^{\infty} \frac{\alpha_k}{k} x^k\bigr)$ produces a polynomial whose roots $\rho_1, \ldots, \rho_n$ satisfy $\sum_{i=1}^n \rho_i^{-k} = \alpha_k$ exactly for $k = 1, \ldots, n$. This truncation-exactness property is an algebraic identity in the ring of formal power series, proved by coefficient matching. It defines a natural embedding of sequences into multisets of complex numbers and yields an $O(n^2)$ algorithm for computing the polynomial from the prescribed power sums. We apply the result to the polylogarithm family $\alpha_k = k^{1-s}$, where the associated exponential $\exp(-\mathrm{Li}_s(x))$ produces factorial-integer coefficient sequences for $s \leq 0$ and encodes values of the Riemann zeta function through $\lim_{n\to\infty} P_n^{(s)}(1) = \exp(-\zeta(s))$ for $\mathrm{Re}(s) > 1$.
\end{abstract}

\subjclass[2020]{Primary 11B83; Secondary 05E05, 11M06}
\keywords{Power sums, plethystic exponential, Newton's identities, polylogarithm, Riemann zeta function}

\maketitle

\section{Introduction}

The study of power sums over roots of polynomials has a long history, from Newton's identities \cite{NewtonIdentities} to modern applications in symmetric function theory \cite{MacdonaldSymmetric, StanleyEC2}. A natural inverse problem arises: given a prescribed sequence $(\alpha_1, \alpha_2, \ldots)$, construct a polynomial whose roots realize these values as power sums. The plethystic exponential provides an elegant solution in the infinite setting \cite{Plethystic}. For a sequence $\boldsymbol{\alpha} = (\alpha_k)_{k\geq 1}$, the generating function

\[
    f(x) = \exp\!\left(-\sum_{k=1}^{\infty} \frac{\alpha_k}{k}\, x^k\right) = \prod_{i=1}^{\infty}(1 - x/\rho_i)
\]

produces roots $\{\rho_i\}$ satisfying $\sum_i \rho_i^{-k} = \alpha_k$ for all $k \geq 1$---this is Newton's identities recast in generating function form \cite{LangAlgebra}. For both computational and theoretical purposes, one needs \emph{finite} polynomial approximations. The natural approach is to truncate $f(x)$ to degree $n$. The question is whether truncation destroys the power sum property. The main result of this paper answers this precisely: truncation preserves the power sum constraints \emph{exactly} for the first $n$ terms. This is an algebraic identity, not an approximation or an asymptotic statement.

\begin{theorem}[Truncation-Exactness]\label{thm:main}

Let $(\alpha_1, \ldots, \alpha_n) \in \mathbb{C}^n$ be arbitrary. Define

\[
    P_n(x) = \left[\exp\!\left(-\sum_{k=1}^{\infty} \frac{\alpha_k}{k}\, x^k\right)\right]_{\deg \leq n},
\]

where $[\,\cdot\,]_{\deg \leq n}$ denotes truncation to degree $n$. If $\rho_1, \ldots, \rho_n$ are the roots of $P_n(x)$ counted with multiplicity, then

\[
    \sum_{i=1}^{n} \rho_i^{-k} = \alpha_k \qquad \text{exactly, for } k = 1, 2, \ldots, n.
\]

\end{theorem}

The key point is that we truncate the \emph{output} (the polynomial), not the \emph{input} (the log-generating series). The proof proceeds by coefficient matching in the formal power series \cite{StanleyEC2} and is given in Section~\ref{sec:proof}. Theorem~\ref{thm:main} also defines an embedding

\[
    \Phi_n : \mathbb{C}^n \longrightarrow \{\text{multisets of } n \text{ points in } \mathbb{C}\},
\]

sending $(\alpha_1, \ldots, \alpha_n)$ to the roots of $P_n(x)$, showing that every finite sequence of complex numbers is realizable as the first $n$ negative power sums of a degree-$n$ polynomial.

The paper is organized as follows. Section~\ref{sec:prelim} fixes notation. Section~\ref{sec:proof} proves Theorem~\ref{thm:main} and develops the embedding and coefficient recurrence. Section~\ref{sec:polylog} examines the polylogarithm family $\alpha_k = k^{1-s}$, connecting to $\exp(-\mathrm{Li}_s(x))$ and values of the Riemann zeta function \cite{TitchmarshZeta, HardyWright}. Section~\ref{sec:discussion} concludes with open questions.

\section{Preliminaries}\label{sec:prelim}

\begin{definition}[Log-generating function]
Given $\boldsymbol{\alpha} = (\alpha_k)_{k \geq 1}$ with $\alpha_k \in \mathbb{C}$, the \emph{log-generating function} is

\[
    g_{\boldsymbol{\alpha}}(x) = \sum_{k=1}^{\infty} \frac{\alpha_k}{k}\, x^k.
\]

\end{definition}

\begin{definition}[Associated exponential and truncated polynomial]
The \emph{associated exponential} is $f_{\boldsymbol{\alpha}}(x) = \exp(-g_{\boldsymbol{\alpha}}(x))$, and its \emph{degree-$n$ truncation} is

\[
    P_n(x) = [f_{\boldsymbol{\alpha}}(x)]_{\deg \leq n} = \sum_{j=0}^{n} a_j x^j,
\]

where $a_j = [x^j]\, f_{\boldsymbol{\alpha}}(x)$ denotes the coefficient of $x^j$ in the power series expansion.
\end{definition}

\section{Proof of the Main Theorem}\label{sec:proof}

We prove Theorem~\ref{thm:main} via coefficient matching in formal power series \cite{StanleyEC2}.

\begin{lemma}[Factored form]\label{lem:factor}
Let $P_n(x) = \sum_{j=0}^n a_j x^j$ with $a_0 = 1$ and roots $\rho_1, \ldots, \rho_n$ counted with multiplicity. Then all roots are nonzero and 

\[
    P_n(x) = \prod_{i=1}^{n} \left(1 - \frac{x}{\rho_i}\right).
\]

\end{lemma}

\begin{proof}

Since $a_0 = P_n(0) = 1 \neq 0$, no root is zero. Writing $P_n(x) = c\prod_{i=1}^n(x - \rho_i)$ and evaluating at $x = 0$ gives $c \cdot (-1)^n \prod_i \rho_i = 1$, which yields the stated factored form.
\end{proof}

\begin{lemma}[Logarithm of factored form]\label{lem:log}
In the ring of formal power series, we have 

\[
    \log P_n(x) = -\sum_{k=1}^{\infty} \frac{p_{-k}}{k}\, x^k,
\]

where $p_{-k} = \sum_{i=1}^n \rho_i^{-k}$ denotes the $k$-th negative power sum of the roots.
\end{lemma}

\begin{proof}
Applying $\log(1 - u) = -\sum_{k=1}^{\infty} u^k/k$ term by term:
\begin{align*}
\log P_n(x) &= \sum_{i=1}^{n} \log\!\left(1 - \frac{x}{\rho_i}\right) = -\sum_{i=1}^{n} \sum_{k=1}^{\infty} \frac{1}{k} \left(\frac{x}{\rho_i}\right)^k \\
&= -\sum_{k=1}^{\infty} \frac{1}{k} \left(\sum_{i=1}^{n} \rho_i^{-k}\right) x^k = -\sum_{k=1}^{\infty} \frac{p_{-k}}{k}\, x^k. \qedhere
\end{align*}
\end{proof}

\begin{proof}[Proof of Theorem~\ref{thm:main}]
Let $h(x) = -\sum_{k=1}^{\infty} \frac{\alpha_k}{k}\, x^k$, so that $P_n(x) = [\exp(h(x))]_{\deg \leq n}$ by definition. In the ring of formal power series, $\exp$ and $\log$ are inverse operations on series with constant term $1$ and $0$ respectively \cite{StanleyEC2}. Therefore, extracting coefficients up to degree $n$,
\[
    [\log P_n(x)]_{\deg \leq n} = [\log[\exp(h(x))]_{\deg \leq n}]_{\deg \leq n} = [h(x)]_{\deg \leq n}.
\]

On the other hand, Lemma~\ref{lem:log} gives $\log P_n(x) = -\sum_{k=1}^{\infty} \frac{p_{-k}}{k}\, x^k$. Comparing coefficients of $x^k$ for $k = 1, \ldots, n$:

\[
    -\frac{p_{-k}}{k} = -\frac{\alpha_k}{k} \implies p_{-k} = \alpha_k,
\]

which is $\sum_{i=1}^n \rho_i^{-k} = \alpha_k$ for $k = 1, \ldots, n$.
\end{proof}

Theorem~\ref{thm:main} has two immediate consequences. First, it defines a natural embedding: the map $\Phi_n : \mathbb{C}^n \longrightarrow \{\text{multisets of } n \text{ points in } \mathbb{C}\}$ sending $(\alpha_1, \ldots, \alpha_n)$ to the roots of $P_n(x)$ is well-defined, and every finite sequence of complex numbers is realizable as the first $n$ negative power sums of a degree-$n$ polynomial. Second, the coefficients of $P_n$ are computable directly from the $\alpha_k$ via a simple recurrence.

\begin{proposition}[Coefficient recurrence]\label{prop:recurrence}
The coefficients $a_0, a_1, \ldots, a_n$ of $P_n(x)$ satisfy $a_0 = 1$ and 

\[
    a_k = -\frac{1}{k} \sum_{j=1}^{k} \alpha_j\, a_{k-j} \qquad \text{for } k = 1, \ldots, n.
\]

\end{proposition}

\begin{proof}
Differentiating $f = \exp(-g)$ gives $f' = -g'f$. With $g'(x) = \sum_{k=1}^{\infty} \alpha_k x^{k-1}$, expanding $f' = -g'f$ and matching coefficients of $x^{k-1}$ yields $k\, a_k = -\sum_{j=1}^{k} \alpha_j\, a_{k-j}$. Dividing by $k$ gives the result.
\end{proof}

This recurrence computes $P_n$ from $(\alpha_1, \ldots, \alpha_n)$ in $O(n^2)$ arithmetic operations.

\section{The Polylogarithm Family}\label{sec:polylog}

A natural and rich source of examples is the family $\alpha_k = k^{1-s}$ for $s \in \mathbb{C}$. In this case the log-generating function is

\[
    g(x) = \sum_{k=1}^{\infty} \frac{k^{1-s}}{k}\, x^k = \sum_{k=1}^{\infty} \frac{x^k}{k^s} = \mathrm{Li}_s(x),
\]

the polylogarithm of order $s$ \cite{LewinPolylog}. The associated exponential is therefore $f(x) = \exp(-\mathrm{Li}_s(x))$, and Theorem~\ref{thm:main} guarantees that truncating $f$ to degree $n$ yields a polynomial whose first $n$ negative power sums equal $k^{1-s}$ exactly. For non-positive integers $s = -m$, the polylogarithm reduces to a rational function \cite{LewinPolylog, EulerianNumbers}

\[
    \mathrm{Li}_{-m}(x) = \frac{x\, A_m(x)}{(1-x)^{m+1}},
\]

where $A_m(x)$ is the $m$-th Eulerian polynomial. Table~\ref{tab:polylog} summarizes the family for $s \in \{-3, \ldots, 3\}$.

\begin{table}[htbp]
\centering
\renewcommand{\arraystretch}{1.7}
\small
\begin{tabular}{c|c|c|l}
\hline
$s$ & $\alpha_k$ & $g(x) = \mathrm{Li}_s(x)$ & $f(x) = \exp(-g(x))$ \\
\hline
$-3$ & $k^4$ & $\dfrac{x(1+4x+x^2)}{(1-x)^4}$ & $1 - x - \dfrac{15}{2!}x^2 - \dfrac{115}{3!}x^3 - \dfrac{215}{4!}x^4 + \cdots$ \\[6pt]
$-2$ & $k^3$ & $\dfrac{x(1+x)}{(1-x)^3}$ & $1 - x - \dfrac{7}{2!}x^2 - \dfrac{31}{3!}x^3 - \dfrac{23}{4!}x^4 + \cdots$ \\[6pt]
$-1$ & $k^2$ & $\dfrac{x}{(1-x)^2}$ & $1 - x - \dfrac{3}{2!}x^2 - \dfrac{7}{3!}x^3 + \dfrac{1}{4!}x^4 + \cdots$ \\[6pt]
$0$  & $k$   & $\dfrac{x}{1-x}$     & $1 - x - \dfrac{1}{2!}x^2 - \dfrac{1}{3!}x^3 + \dfrac{1}{4!}x^4 + \cdots$ \\[6pt]
$1$  & $1$   & $-\log(1-x)$         & $1 - x$ \\[6pt]
$2$  & $1/k$ & $\mathrm{Li}_2(x)$   & $1 - x + \dfrac{1}{4}x^2 - \dfrac{1}{36}x^3 - \dfrac{1}{288}x^4 + \cdots$ \\[6pt]
$3$  & $1/k^2$ & $\mathrm{Li}_3(x)$ & $1 - x + \dfrac{3}{8}x^2 - \dfrac{17}{216}x^3 + \dfrac{29}{3456}x^4 + \cdots$ \\
\hline
\end{tabular}
\vspace{3mm}
\caption{The polylogarithm family $f(x) = \exp(-\mathrm{Li}_s(x))$ for integer $s$. The case $s = 1$ is exact: $f(x) = 1 - x$. For $s \leq 0$, coefficients are written as $a_n = b_n/n!$ with integer $b_n$.}
\label{tab:polylog}
\end{table}

The case $s = 1$ is degenerate in the best sense: since $\mathrm{Li}_1(x) = -\log(1-x)$, we get $f(x) = \exp(\log(1-x)) = 1 - x$ exactly, with single root $\rho = 1$ satisfying $\rho^{-k} = 1 = \alpha_k$ for all $k$. For $s \leq 0$, the factorial-scaled coefficients $b_n = n!\,a_n$ are integers. The sequences for $s = 0$ and $s = -1$ appear in the OEIS as A293116 and A318215, respectively \cite{OEIS}.

\subsection*{Detailed verification: the case $s = 0$}

We illustrate Theorem~\ref{thm:main} concretely for $\alpha_k = k$. The log-generating function is $g(x) = x/(1-x)$, giving
\[
    f(x) = \exp\!\left(-\frac{x}{1-x}\right) = 1 - x - \frac{x^2}{2} - \frac{x^3}{6} + \frac{x^4}{24} + \frac{19x^5}{120} + \cdots
\]
The first few truncations are
\begin{align*}
P_2(x) &= 1 - x - \tfrac{1}{2}x^2, & P_4(x) &= 1 - x - \tfrac{1}{2}x^2 - \tfrac{1}{6}x^3 + \tfrac{1}{24}x^4, \\
P_3(x) &= 1 - x - \tfrac{1}{2}x^2 - \tfrac{1}{6}x^3, & P_5(x) &= 1 - x - \tfrac{1}{2}x^2 - \tfrac{1}{6}x^3 + \tfrac{1}{24}x^4 + \tfrac{19}{120}x^5.
\end{align*}
Table~\ref{tab:verification} verifies the power sum property numerically: the entry $\sum_i \rho_i^{-k}$ for $P_n$ equals $\alpha_k = k$ exactly when $k \leq n$ (bold), and deviates for $k > n$, in precise agreement with Theorem~\ref{thm:main}.

\begin{table}[htbp]
\centering
\renewcommand{\arraystretch}{1.3}
\begin{tabular}{c|cccccc}
\hline
& $k=1$ & $k=2$ & $k=3$ & $k=4$ & $k=5$ & $k=6$ \\
\hline
$P_2$ & $\mathbf{1}$ & $\mathbf{2}$ & $2.50$ & $3.50$ & $4.75$ & $6.50$ \\
$P_3$ & $\mathbf{1}$ & $\mathbf{2}$ & $\mathbf{3}$ & $4.17$ & $6.00$ & $8.58$ \\
$P_4$ & $\mathbf{1}$ & $\mathbf{2}$ & $\mathbf{3}$ & $\mathbf{4}$ & $5.79$ & $8.21$ \\
$P_5$ & $\mathbf{1}$ & $\mathbf{2}$ & $\mathbf{3}$ & $\mathbf{4}$ & $\mathbf{5}$ & $7.26$ \\
\hline
Target $\alpha_k$ & $1$ & $2$ & $3$ & $4$ & $5$ & $6$ \\
\hline
\end{tabular}
\caption{Inverse power sums $\sum_i \rho_i^{-k}$ for the roots of $P_n(x)$ with $\alpha_k = k$. Bold entries indicate exact agreement with $\alpha_k$, occurring precisely for $k \leq n$.}
\label{tab:verification}
\end{table}

\begin{figure}[htbp]
\centering
\includegraphics[width=\textwidth]{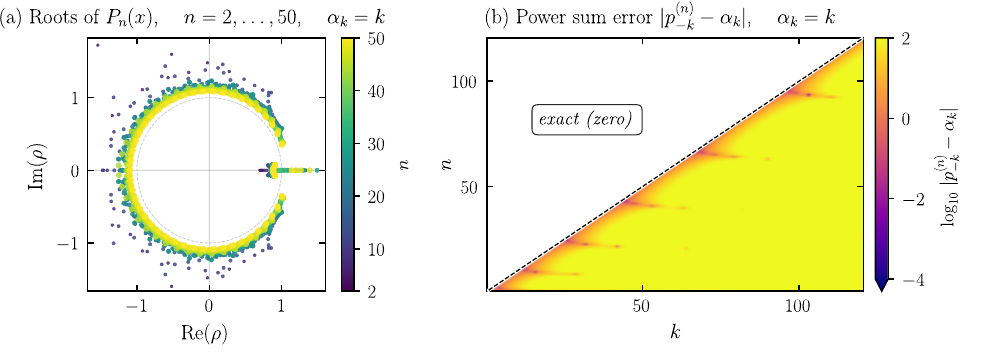}
\caption{(a) Roots of $P_n(x) = [\exp(-x/(1-x))]_{\deg \leq n}$ in the complex plane for $n = 2, \ldots, 50$, colored by $n$ (purple: small $n$, yellow: $n = 50$). (b) Heatmap of $\log_{10}|p_{-k}^{(n)} - \alpha_k|$ for $\alpha_k = k$, $n, k = 1, \ldots, 120$. White pixels (lower-left triangle, $k \leq n$) indicate errors below machine precision ($< 10^{-8}$), confirming Theorem~\ref{thm:main}.}
\label{fig:main}
\end{figure}

For $\mathrm{Re}(s) > 1$, the polylogarithm satisfies $\mathrm{Li}_s(1) = \zeta(s)$ \cite{TitchmarshZeta}, which gives a concrete evaluation of the limiting exponential.

\begin{corollary}\label{cor:zeta}
For $\mathrm{Re}(s) > 1$, 
\[
\lim_{n \to \infty} P_n^{(s)}(1) = \exp(-\zeta(s)).
\]
\end{corollary}

\begin{proof}
As $n \to \infty$, the truncated polynomial $P_n^{(s)}(x)$ converges coefficient-wise to $f(x) = \exp(-\mathrm{Li}_s(x))$. Evaluating at $x = 1$ and using $\mathrm{Li}_s(1) = \zeta(s)$ gives the result.
\end{proof}

Thus the truncated polynomials encode values of the Riemann zeta function through their evaluation at $x = 1$, providing a polynomial approximation scheme for $\exp(-\zeta(s))$.

\section{Discussion}\label{sec:discussion}

Theorem~\ref{thm:main} establishes a fundamental algebraic identity: the degree-$n$ truncation of $\exp(-g_{\boldsymbol{\alpha}}(x))$ produces a polynomial whose roots satisfy the first $n$ negative power sum constraints exactly. This is not an approximation result but a consequence of the formal power series identity $[\log[\exp(h)]_{\leq n}]_{\leq n} = [h]_{\leq n}$. The associated embedding $\Phi_n : \mathbb{C}^n \to \{\text{multisets in } \mathbb{C}\}$ shows that every finite sequence of complex numbers is realizable as negative power sums of a degree-$n$ polynomial, computable in $O(n^2)$ operations via Proposition~\ref{prop:recurrence}.

The polylogarithm family $\alpha_k = k^{1-s}$ illustrates the theorem concretely across a wide range of behavior: from the trivial case $s = 1$ (single root $\rho = 1$) to the factorial-integer sequences arising for $s \leq 0$, to the zeta value encoding of Corollary~\ref{cor:zeta}. The heatmap in Figure~\ref{fig:main}(b) makes the theorem visible at scale: a perfectly sharp triangular region of machine-zero error, bounded by the diagonal $k = n$.

Several directions remain open. First, the root distribution question: numerical evidence (Figure~\ref{fig:main}(a)) suggests that roots of $P_n(x)$ accumulate on a smooth curve in $\mathbb{C}$ as $n \to \infty$; a rigorous description of this limit curve is open. Second, the factorial-scaled integer sequences $b_n = n!\,a_n$ arising for $s \leq 0$ appear to be new for $s \leq -2$; their combinatorial interpretation and asymptotic growth rates are unresolved. Third, the rate of convergence of $P_n^{(s)}(1)$ to $\exp(-\zeta(s))$ for $s$ near the critical strip deserves a precise analytic treatment.

\bibliographystyle{amsplain}

\begin{thebibliography}{10}

\bibitem{EulerianNumbers}
T.~K. Petersen, \emph{Eulerian Numbers}, Birkh\"auser, New York, 2015.

\bibitem{HardyWright}
G.~H. Hardy and E.~M. Wright, \emph{An Introduction to the Theory of Numbers}, 6th ed., Oxford University Press, Oxford, 2008.

\bibitem{LangAlgebra}
S.~Lang, \emph{Algebra}, 3rd ed., Springer, New York, 2002.

\bibitem{LewinPolylog}
L.~Lewin, \emph{Polylogarithms and Associated Functions}, North-Holland, New York, 1981.

\bibitem{MacdonaldSymmetric}
I.~G. Macdonald, \emph{Symmetric Functions and Hall Polynomials}, 2nd ed., Oxford University Press, Oxford, 1995.

\bibitem{NewtonIdentities}
I.~Newton, \emph{Arithmetica Universalis}, Typis Academicis, Cambridge, 1707; English translation by J.~Raphson, London, 1720.

\bibitem{OEIS}
N.~J.~A. Sloane, \emph{The On-Line Encyclopedia of Integer Sequences}, \url{https://oeis.org}, 2024.

\bibitem{Plethystic}
B.~Feng, A.~Hanany, and Y.-H. He, Counting gauge invariants: the plethystic program, \emph{J. High Energy Phys.} \textbf{2007} (2007), no.~03, 090.

\bibitem{StanleyEC2}
R.~P. Stanley, \emph{Enumerative Combinatorics, Vol.~2}, Cambridge University Press, Cambridge, 1999.

\bibitem{TitchmarshZeta}
E.~C. Titchmarsh, \emph{The Theory of the Riemann Zeta-Function}, 2nd ed., Oxford University Press, Oxford, 1986.

\end{thebibliography}

\end{document}